\documentclass[12pt,oneside]{amsart}
     \usepackage{amssymb}
    \usepackage{amsmath}
    \usepackage[english]{babel}
    \usepackage[autostyle]{csquotes}
    \usepackage{hyperref}

      \theoremstyle{plain}
      \newtheorem{theorem}{Theorem}[section]
      \newtheorem{lemma}[theorem]{Lemma}
      \newtheorem{corollary}[theorem]{Corollary}
      \newtheorem{proposition}[theorem]{Proposition}

      \theoremstyle{definition}
      \newtheorem{definition}[theorem]{Definition}
      
      \newtheorem{remark}{Remark}
       \newtheorem{example}{Example}[section]
     \newtheorem*{notation}{Notation}   
      \newtheorem{question}[theorem]{Question}

      \newcommand{\nil}{\emptyset}
      \newcommand{\g}{{\mathfrak g}}
      \newcommand{\h}{{\mathfrak{h}}}
      \newcommand{\f}{\mathfrak f}

      \newcommand{\e}{\mathbf e}
       \newcommand{\E}{{\mathcal{E}}}
       \newcommand{\F}{{\mathcal{F}}}
      
     \newcommand{\Con}{\operatorname{Con}}
      \newcommand{\n}[1]{\ell(#1)}
      \newcommand{\con}[2]{\Con(\mathfrak{#1},\mathcal{#2})}

      \newcommand{\T}[2]{\vartheta^{\mathfrak{#1}}_{\mathfrak{#2}}}
       \newcommand{\atom}[1]{\textbf{atom}(\mathcal{#1})} 
     \newcommand{\inn}[1]{\text{Inn}\,#1}
      \newcommand{\A}{\mathcal{A}}
       \newcommand{\B}{\mathcal{B}}

      \makeatletter
      \def\@setcopyright{}
      \def\serieslogo@{}
      \makeatother
   
\begin{document}

   \author{Ali Rejali}
   \address{University of Isfahan}  
   \curraddr{Department of Mathematics, Faculty of  Sciences , University of Isfahan University, Isfahan, 81746-73441, Iran}
   \email{rejali@sci.ui.ac.ir}


   \author{Meisam Soleimani Malekan}
   \address{Department of Mathematics, Ph. D student, University of Isfahan, Isfahan, 81746-73441, Iran}
   \email{m.soleimani@sci.ui.ac.ir}

  \title{Solubility of groups can be characterized by configuration}

\begin{abstract} The concept of configuration was first introduced by Rosenblatt
and Willis to give a characterization for the amenability of groups. We show that group properties of being soluble or FC can be characterized by configuration sets. Then we investigate some condition on configuration pairs, which leads to isomorphism. We introduce a somewhat different notion of configuration equivalence, namely strong configuration equivalence, and prove that strong configuration equivalence coincides with isomorphism.  \end{abstract}

   \subjclass[2010]{20F05, 20F16, 20F24, 43A07}

   \keywords{Configuration, Amenability, Commutator Subgroup, FC-Group, Soluble Group, Finitely Presented Group, Hopfian Group}

   \thanks{The authors would like to express  their gratitude towards Banach Algebra Center of Excellence for Mathematics, University of Isfahan.}
   \thanks{}

   \dedicatory{}

   \date{\today}


   \maketitle



\section{Introduction and Definitions}
    
   In the present paper, all groups are assumed to be finitely generated. Let $G$ be a group, we denote the identity of the group $G$ by $e_G$. We refer readers to \cite{combin} for terminology and statements used for finitely generated groups. \par     
    The notion of a configuration for a group was introduced in \cite{rw}. It was shown in that paper that the amenability of a group can be characterized by configurations.   
  \begin{definition}
  \label{1}
  let $G$ be a group. Let $\mathfrak{g} =(g_1,\dotsc, g_n)$ be an ordered generating set and $\mathcal{E} =\{E_1,\dotsc,E_m\}$ be a finite partition of G. 
  \par A \textit{configuration} $C$ corresponding to $(\g,\E)$, is an $(n+1)$-tuple $C=(c_0,\dotsc, c_n)$, where $c_k\in\{1,\dotsc,m\}$ for each $k$, such that there are $x_0, x_1,\dotsc, x_n\in G$ with $x_k\in E_{c_k}$, $j=0,1,\dotsc, n$, and for each $k=1,\dotsc,n$, $x_k=g_kx_0$. In this case, we say that $(x_0,x_1,\dotsc,x_n)$ has configuration $C$.
  \end{definition}
  For $\g$ and $\E$ as above, we call $(\g,\E)$ a \textit{configuration pair}. The set of configurations corresponding to the configuration pair $(\g,\E)$ will be denoted by $\con gE$. The set of all configuration sets of $G$ is denoted by $\Con(G)$. It is not hard to see,
\begin{remark}\label{remark 1}
 Let $\con gE$ be a configuration set for a group $G$ and let us have $y\in G$ and $E\in\E$. Then it may be assumed that $y\in E$.
\end{remark}
 In \cite{rw}, the authors conjectured that combinatorial properties of configurations can be used to characterize various kinds of behavior of groups, specially, group properties which lead to amenability. According to this conjecture, in [2], the notion of configuration equivalence was created: A group $G$ is \textit{configuration contained} in a group $H$, written $G\precsim H$, if $\Con(G)\subseteq\Con(H)$, and two groups $G$ and $H$ are \textit{configuration equivalent}, written $G\approx H$, if $\Con(G) =\Con(H)$. \par 
 It would be worthy of mention that the condition that $\Con(\g,\E)=\Con(\h,\F)$ implies that the generating sets $\g$ and $\h$ and the partitions $\E$ and $\F$ each have the same numbers of elements.
 \begin{notation}\label{Os}
 Let $G$ and $H$ be two groups with generating sets $\g$ and $\h$, respectively. Suppose that for partitions 
 $\E=\{E_1,\dotsc,E_m\}$ and $\F=\{F_1,\dotsc,F_m\}$ of $G$ and $H$ respectively, the equality 
 $\con gE=\con hF$ established. Then we say that $E_i$ is \textit{corresponding} to $F_i$, and write $E_i \leftrightsquigarrow F_i$, $i=1,\dotsc,m$.
 \end{notation} 
 
 The first question discussed following the definition of configuration equivalence is that of which properties of the groups can be characterized by configuration sets?
\par In \cite{arw}, Abdollahi, Rejali and Willis showed that finiteness and periodicity are the properties which can be characterized by configuration. In that paper, the authors proved that for two configuration equivalent groups, the classes of their isomorphic finite quotients are the same. The word "finite" in the previous statement, can be replaced by "Abelian" (see \cite{ary}). 
\\ Let $\textbf{F}_n$ be the free group on the set $\{f_1,\dotsc,f_n\}$, where $n$ is a positive integer. Suppose that $\mu=\mu(f_1,\dotsc,f_n)$ is an element of $\textbf{F}_n$. we call $\mu=e_G$ a group-law in a group $G$, if for all $n$-tuples $(x_1,\dotsc,x_n)$ of elements of $G$, we have $\mu(x_1,\dotsc,x_n)=e_G$. It was shown in \cite[Theorem 5.1]{arw} that two configuration equivalent groups, should satisfy in the same semi-group laws, and we generalized this result by proving that same group laws should be established in configuration equivalent groups. Hence, in particular, being Abelian and the group property of being nilpotent of class $c$ are other properties which can be characterized by configuration (see \cite{arw} and \cite{ary}). In \cite{ary}, it was shown that if $G\approx H$, and $G$ is a torsion free nilpotent group of Hirsch length $h$, then so is $H$. It is interesting to know the answer to the question whether being FC-group is conserved by equivalence of configuration. In \cite{ary}, this question was answered under the assumption of being-nilpotent. Here we affirmatively answer this question without any extra hypothesis. In addition, we show that the solubility of a group $G$ can be recovered from $\Con(G)$. 
\par 
 Also, the question that in which groups configuration equivalence implies isomorphism, has been of interest. In other words, for which groups $G$, if $G\approx H$ for a group $H$, then will $H$ be isomorphic to $G$? 
\par In \cite{arw}, it was shown that for the classes of finite, free and Abelian groups, these two notions, configuration equivalence and isomorphism, are the same. In \cite{ary}, it was proved that those groups with the form of $\mathbb{Z}^n\times F$, where $\mathbb Z$ is the group of integers, $n$ is a positive integer and $F$ is an arbitrary finite group, are determined up to isomorphism by their configuration sets. In \cite{ary}, it was proved that if $G\approx D_\infty$, where $D_\infty$ is the infinite dihedral group, then $G\cong D_\infty$.
 \par Studying the proof of the statements mentioned in \cite{arw} and \cite{ary}, we found out that it was the existence of certain configuration pairs which implied isomorphism. We call this certain type of configuration pair \textit{golden} and in Theorem \ref{last}, we will show that in the class of finitely presented Hopfian groups with golden configuration pair, configuration equivalence coincides with isomorphism.
\par For the concept of configuration equivalence matches with isomorphism, we think that the identity element of a group should be recognized by configuration sets, and it seems that the usual definition of configuration equivalence could not do so; That is, if a partition $\E$ of a group $G$ contains $\{e_G\}$, then the equality $\con gE=\con {g'}{E'}$, for configuration pairs $(\g,\E)$ and $(\g',\E')$ of $G$, can not assure us that $\E'$ contains $\{e_G\}$, too. This defect propelled us to introduce a new version of configuration equivalence which turns to be coincided with isomorphism.

\section{Configuration and Group Properties}
 At first, we require the notation below to avoid writing long in our proofs:   
  \begin{notation}
  \label{not}
Let $G$ be a group with $\mathfrak{g}=(g_1,\dotsc,g_n)$ as its generating set. Let $p$ be a positive integer, let  $J$ and $\rho$ be $p$-tuple with components in $\{1,2,\dotsc,n\}$ and $\{\pm1\}$, respectively. We denote the product 
$\prod_{i=1}^pg_{J(i)}^{\rho(i)}$ by $W(J,\rho;\mathfrak{g})$. We call the pair 
$(J,\rho)$ a \textit{representative pair} on $\g$ and $W(J,\rho;\mathfrak{g})$ the \textit{word} corresponding to $(J,\rho)$ in $\g$. 
\end{notation}
For an arbitrary multiple, $J$, we denote the number of its components by $\n J$. When we speak of a representative pair, $(J,\rho)$, we assume the same number of components for $J$ and $\rho$. 
If $J=(J(1),\dotsc,J(p))$, and $\rho=(\rho(1),\dotsc,\rho(p))$, where $p$ is a positive integer, we set
 \begin{align*}
 J^{-1}&:=(J(p),\dotsc,J(1))\quad\text{and}\quad
 \rho^{-1}:=(-\rho(p),\dotsc,-\rho(1))
\end{align*}
\par
For $p_i\in\mathbb{N}$, if $J_i$ is a $p_i$-tuple, $i=1,2$, $J_1\oplus J_2$ is a $(p_1+p_2)$-tuple that has $J_1$ as its first $p_1$ components, and $J_2$ as its second $p_2$ components. It can be easily seen that
$$W(J_1,\rho_1;\mathfrak{g})W(J_2,\rho_2;\mathfrak{g})=
W(J_1\oplus J_2,\rho_1\oplus\rho_2;\mathfrak{g})$$
and 
$$W(J,\rho;\g)^{-1}=W(J^{-1},\rho^{-1};\g)$$
for representative pairs $(J_i,\rho_i)$, $i=1,2$.
\par Let $G$ and $H$ be two groups with generating sets  $\g=(g_1,\dotsc,g_n)$ and $\h=(h_1,\dotsc,h_n)$, respectively. There is a relation, denoted by $\T hg$, from $G$ to $H$ which contains $(g,h)\in G\times H$, if there is a representative pair $(J,\rho)$ such that $g=W(J,\rho;\g)$ and $h=W(J,\rho;\h)$. By the above notation, it is easily noticeable that:
\begin{itemize}
\item If $\T hg$ is a function, then it will automatically be a homomorphism.
\item $\T hg$ is an epimorphism of groups if and only if for every representative pair $(J,\rho)$, $W(J,\rho;\g)=e_G$ implies $W(J,\rho;\h)=e_H$.
\item  $\T hg$ is an isomorphism of groups if and only if both relations, $\T hg$ and $\T gh$ are epimorphism. 
\end{itemize}
\par Recall that we say that a property \textbf{P} can be characterized by configuration sets if all of configuration equivalent groups have property \textbf{P} in common or do not have this property. It is likely that the group properties which imply amenability, can be characterized by configurations. In the papers written on configuration, some of these properties such as being finite, Abelian, nilpotent of class $c$, amenable or non-amenable are investigated. We will prove that being FC and solubility are two other such properties that can be characterized by configurations.\\ 
\par In the definition of configuration sets, we can replace "partition" by a " finite $\sigma$-algebra". Working with finite $\sigma$-algebras save us writing long. We involve $\sigma$- algebras in the theory of configuration as follows:\par 
 Let $G$ be a group. There is a correspondence between finite $\sigma$-algebras of $G$, and finite partitions of $G$. Indeed, for a finite $\sigma$-algebra $\A$, the set of atoms\footnote{An atomic set of a $\sigma$-algebra $\A$, is a non-empty element, which contains no other elements of $\A$.} of $\A$ is a partition of $G$, and for a finite collection $\mathcal{C}$ of subsets of $G$, the $\sigma$-algebra generated by elements of $\mathcal{C}$ is finite. We denote the atomic sets of a $\sigma$-algebra $\A$ by $\atom A$. Also, if $\mathcal{C}$ is a finite collection of subsets of $G$, we use $\sigma(\mathcal C)$ to denote the $\sigma$-algebra generated by $\mathcal C$. In the following, we always consider sigma algebras to be finite. Now, for a $\sigma$-algebra $\A$, we define $\con gA$ to be $\Con(\g,\atom A)$. \\
We can also use $\leftrightsquigarrow$ for sigma algebras; Let $\E:=\{E_1,\dotsc,E_m\}$ and $\F:=\{F_1,\dotsc,F_m\}$ be partitions of $G$ and $H$ respectively, such that $E_i\leftrightsquigarrow F_i$, $i=1,\dotsc,m$. For $A\in\sigma(\E)$ and $B\in\sigma(\F)$, say $A\leftrightsquigarrow B$, when 
$$\{k:\,E_k\cap A\neq\nil\}=\{k:\, F_k\cap B\neq\nil\}.$$
In other words, if $A\leftrightsquigarrow B$, and $A=E_{i_1}\cup\dotsb\cup E_{i_j}$, then $B=F_{i_1}\cup\dotsb\cup F_{i_j}$. 
In the following, we will use this technical lemma:
\begin{lemma}
\label{important}
Let $G$ and $H$ be two groups with finite $\sigma$-algebras $\A$ and $\B$, and generating sets $\g=(g_1,\dotsc,g_n)$, and $\h=(h_1,\dotsc,h_n)$, respectively, such that $\con gA=\con hB$. Suppose that $A_1,A_2\in\A$ and $B_1, B_2\in\B$, are such that $A_i\leftrightsquigarrow B_i$, $i=1,2$. we have 
\begin{itemize}
\item[(a)] If $g_rA_1\subseteq A_2$, then $h_rB_1\subseteq B_2$, $r\in\{1,\dotsc,n\}$,
\item[(b)] If $g_rA_1=A_2$, then $h_rB_1=B_2$, $r\in\{1,\dotsc,n\}$.
\end{itemize}
\end{lemma}
\begin{proof}
Set 
$$\atom A=\{E_1,\dotsc,E_m\}, \text{and}\,\atom B=\{F_1,\dotsc,F_m\}.$$
such that $E_i\leftrightsquigarrow F_i$, $i=1,\dotsc,m$. Without loss of generality, assume that $r=1$. Also, set 
$$‎I_k:=\{i: ‏‎E‎_i\cap ‎A_k\neq\nil\},\quad ‎k=1,2.$$
So, by assumptions, 
$$A_k=\bigcup_{i\in I_k}E_i, \text{and}\, B_k=\bigcup_{i\in I_k}F_i\quad(k=1,2).$$
Now, for $‎‎‎‎‎‎C=‎‎‎(c_0,c_1,\dotsc,c_n)‎$ in $\con gA$, $c_1\in I_2$ if $c_0\in I_1$, this proves (a).\par 
For proving (b), note that if $‎‎‎‎‎‎C=‎‎‎(c_0,c_1,\dotsc,c_n)‎$ is in $\con gA$, then $c_0\in I_1$, if and only if $c_1\in I_2$.
\end{proof}
 A little more preparation is needed to go through the main lemma of this paper:
 \begin{definition}
Assume that $G$ and $H$ are two groups, and let $F$ be a finite subset of $G$ containing $e_G$. A map $\phi:FF^{-1}\rightarrow H$ is called \textit{a local homomorphism} on $F$, if 
$$\phi(xy^{-1})=\phi(x)\phi(y)^{-1}\quad (x,y\in F).$$
 \end{definition}
 Like homomorphisms, for a local homomorphism $\phi$ on $F$, we have
 \begin{itemize}
 \item $\phi(e_G)=e_H$, 
 \item $\phi(x^{-1})=\phi(x)^{-1}$, $x\in F$.
 \end{itemize}
If $F$ is a finite subgroup of $G$, then it will be clear that a local homomorphism $\phi$ on $F$ becomes a homomorphism of groups. 
\par We, now state the key lemma of the paper.
\begin{lemma}
\label{l21}
Let $G$ and $H$ be two groups such that $G\precsim H$. Let $\g=(g_1,\dotsc,g_n)$ be a generating set of $G$ and $\mathfrak{F}$ be a finite set of representative pairs on $\g$. Then there exists a generating set $\h=(h_1,\dotsc,h_n)$ of $H$ such that $\T gh$ is a local homomorphism on $\{e_H\}\cup\{W(J,\rho;\h):(J,\rho)\in\mathfrak{F}\}$ 
\end{lemma}
\begin{proof}
For a representative pair $(J,\rho)$, set $E(J,\rho):=\{W(J,\rho;\g)\}$ and set $E(1)=\{e_G\}$. Let $n_0:=\max\{\n J:\,(J,\rho)\in\mathfrak F\}$ and put 
\begin{small}
\begin{align*}
S_0&:=\{(J,\rho):\, \n J\leq 3 n_0,\rho\,\text{is arbitrary}\}\\
S_1&:=\{(J,\rho):\, \n J\leq 2 n_0,\rho\,\text{is arbitrary}\}\\
S_2&:=\{(J,\rho):\, \n J\leq n_0,\rho\,\text{is arbitrary}\}.
\end{align*}
\end{small}
A combinatorial argument shows that all above sets are finite. Let $\A$ be the $\sigma$-algebra generated by $E(1)$ and the sets $E(J,\rho)$, $(J,\rho)\in S_0$. Since $E(J,\rho)$'s are singleton, we have 
\begin{equation}
\label{ce}
E(J,\rho)=W(J,\rho;\g)E(1),\quad (J,\rho)\in S_0
\end{equation} 
By $G\precsim H$, there is a generating set $\h$ and a $\sigma$-algebra $\B$ of $H$ such that $\con gA=\con hB$. We denote by $F(1)$ and $F(J,\rho)$, $(J,\rho)\in S_0$, elements in $\B$ where
$$E(1)\leftrightsquigarrow F(1), \quad E(J,\rho)\leftrightsquigarrow F(J,\rho),\quad  (J,\rho)\in S_0.$$
 Without loss of generality, we can assume that $e_H\in F(1)$. We claim that the following equations are established 
 \begin{align}
 \label{eq1}
 F(J,\rho)=W(J,\rho;\h)F(1),\quad (J,\rho)\in S_1
 \end{align}
We prove this claim by induction on $\n J$. If $J$ has only one component, there is nothing to be proved by Lemma \ref{important}(b). Now, suppose that the equation (\ref{eq1}) is established when $\n J<p$. Let 
$$J=(J(1),J(2),\dotsc,J(p))\quad\text{and}\quad\rho=(\rho(1),\rho(2),\dotsc\rho(p))$$ be such that $(J,\rho)\in S_1$. Let $I_1=(J(1))$, $I_2=(J(2),\dotsc,J(p))$, $\delta_1=(\rho(1))$, and $\delta_2=(\rho(2),\dotsc,\rho(p))$. Therefore $J=I_1\oplus I_2$ and $\rho=\delta_1\oplus\delta_2$. By induction hypothesis, we have $F(I_2,\delta_2)=W(I_2,\delta_2;\h)\,F(1)$. The equality $\con gA=\con hB$ and Lemma \ref{important}(b) imply that $$F(J,\rho)=W(I_1,\delta_1;\h) F(I_2,\delta_2).$$ So, again using Lemma \ref{important}(b), we have
\begin{small}
\begin{equation*}
\begin{split}
F(J,\rho)&=W(I_1,\delta_1;\h) F(I_2,\delta_2)\\
&=W(I_1,\delta_1;\h)W(I_2,\delta_2;\h) F(1)=W(J,\rho;\h)F(1)
\end{split}
\end{equation*}
\end{small}
and this proves the equation (\ref{eq1}) for $\n J=p$. \par 
By equation (\ref{eq1}) we have $W(J,\rho;\h)\in F(J,\rho)$. Now, if we have $W(J,\rho;\h)=e_H$ for some pair $(J,\rho)\in S_1$, according to obtained equalities, we get $F(J,\rho)=F(1)$, so $E(J,\rho)=E(1)$ and this gives $W(J,\rho;\g)=e_G$. Hence, $\T gh$ is a well-defined local homomorphism on 
$$\{e_H\}\cup\{W(J,\rho;\h): (J,\rho)\in S_2\}.$$
Indeed if $W(J,\rho;\h)=W(I,\delta;\h)$, for $(J,\rho)$ and $(I,\delta)$ in $S_2$, then $W(J\oplus I^{-1},\rho\oplus\delta^{-1};\h)=e_H$, and $(J\oplus I^{-1},\rho\oplus\delta^{-1})\in S_1$, so $W(J\oplus I^{-1},\rho\oplus\delta^{-1};\g)=e_G$, and this implies that $W(J,\rho;\g)=W(I,\delta;\g)$. 
But $\mathfrak{F}\subseteq S_2$, therefore $\T gh$ is a local homomorphism on $\{e_H\}\cup\{W(J,\rho;\h):(J,\rho)\in\mathfrak{F}\}$.
\end{proof}
The following result can be obtained from the proof of the above lemma:
\begin{remark}\label{local homomorphism and refinements}
Let $G$ and $H$ be two groups with $G\approx H$. Let $(\g,\E)$ be a configuration pair of $G$ and $\mathfrak{F}$ be a finite set of representative pairs on $\g$. Let $\E'$ be a refinement of $\E$ which contains $\{e_G\}$ and singletons $\{W(J,\rho;\g)\}$, $(J,\rho)\in S_0$, where $S_0$ is defined as in the proof of the previous lemma. Assume that $\con g{E'}=\con h{F'}$, for a configuration pair $(\h,\F')$ of $H$. We denote by $F(1)$ and $F(J,\rho)$, $(J,\rho)\in S_0$, elements in $\F'$ where
$$\{e_G\}\leftrightsquigarrow F(1), \quad \{W(J,\rho;\g)\}\leftrightsquigarrow F(J,\rho),\quad  (J,\rho)\in S_0.$$
Then we have $W(J,\rho;\h)\in F(J,\rho)$, for $(J,\rho)\in\mathfrak{F}$
\end{remark}
In \cite[Theorem 5.1]{arw}, it was proved that two configuration equivalent groups satisfy in same semi-group laws; Considering Lemma \ref{l21}, we can generalize this result:

\begin{proposition}
Let $G$ and $H$ be two groups with  $H\precsim G$ and suppose that $G$ satisfies the \textbf{group law} $\mu(x_1, . . . , x_n) =e_G$. Then $H$ satisfies the same law.
\end{proposition}

\begin{proof}
Suppose that {\footnotesize $\mu(x_1, . . . , x_n)=\prod_{i=1}^Nx_{J(i)}^{\rho(i)}$} for $N$-tuples $J$ and $\rho$ with {\footnotesize $J\in\{1,2,\dotsc,n\}^N$} and $\rho\in\{\pm1\}^N$. Also, suppose that $H$ does not satisfy in this group law, so there exists $h_1,\dotsc,h_n\in H$, such that $\mu(h_1,\dotsc,h_n)\neq e_H$. Let $\h_0$ be a generating set of $H$, so that $\h=(h_1,\dotsc,h_n)\oplus\h_0$ is also a generating set. By Notation \ref{not}, $W(J,\rho;\h)\neq e_H$, and by the above lemma, we can get a generating set $\g$ of $G$ such that $W(J,\rho;\g)\neq e_G$. This means that $\mu(g_1,\dotsc,g_n)\neq e_G$, which contradicts the group law in $G$.  
\end{proof}

Let $G$ be a group with a generating set $\g=(g_1,\dotsc,g_n)$. We say that representative pair $(J,\rho)$ on $\g$ is in \textit{$k$th derivation form} if, for the free non-Abelian group of rank $n>0$, $\textbf{F}_n$, with generating set $\f=(f_1,\dotsc,f_n)$, $W(J,\rho;\f)\in\textbf{F}_n^{(k)}$, in which the power $(k)$ stands for denoting the $k$th derived subgroup. We have: 
\begin{lemma}
Let $G$ be a group with a generating set $\g=(g_1,\dotsc,g_n)$. Then 
$$G^{(k)}=\{W(J,\rho;\g): (J,\rho)\,\text{is a representative pair in $k$th derivation form}\}.$$ 
\end{lemma}
\begin{proof}
Let $\f$ be a generating set of $\textbf F_n$. Since there are no relations in $\mathbb F_n$, the equality $W(J_1,\rho_1;\f)=W(J_2,\rho_2;\f)$ implies $W(J_1,\rho_1;\g)=W(J_2,\rho_2;\g)$ for the generating set $\g$. \par
For representative pairs $(J,\rho)$ and $(I,\delta)$, we denote $J^{-1}\oplus I^{-1}\oplus J\oplus I$ and $\rho^{-1}\oplus \delta^{-1}\oplus \rho\oplus \delta$ by $[J,I]$ and $[\rho,\delta]$, respectively. By these notations, 
$$[W(J,\rho;g),W(I,\delta;\g)]=W([J,I],[\rho,\delta];\g)$$
where $[x,y]=x^{-1}y^{-1}xy$, $x,y\in G$.\par 
We only prove the lemma in the case where $k=1$. For larger values of $k$ one can use induction. First, suppose that $g\in G^{(1)}$; so there are representative pairs $(J_i,\rho_i)$ and $(I_i,\delta_i)$, $i=1,\dotsc,m$, such that 
\begin{align*}
g&=\prod_{i=1}^m[W(J_i,\rho_i;\g),W(I_i,\delta_i;\g)]\\
&=\prod_{i=1}^mW([J_i,I_i],[\rho_i,\delta_i];\g)\\
&=W\left(\bigoplus_{i=1}^m[J_i,I_i],\bigoplus_{i=1}^m[\rho_i,\delta_i];\g\right);
\end{align*}
but it is clear that $(\bigoplus_{i=1}^m[J_i,I_i],\bigoplus_{i=1}^m[\rho_i,\delta_i])$ is in the first derivation form.
\par Conversely, suppose that $(J,\rho)$ is in the first derivation form, so, by an argument as above, we have 
$$W(J,\rho;\f)=W\left(\bigoplus_{i=1}^m[J_i,I_i],\bigoplus_{i=1}^m[\rho_i,\delta_i];\f\right)$$
for representative pairs $(J_i,\rho_i)$ and $(I_i,\delta_i)$, $i=1,\dotsc,m$. By the note mentioned at the beginning of the proof, the following holds:
\begin{align*}
W(J,\rho;\g)&=W\left(\bigoplus_{i=1}^m[J_i,I_i],\bigoplus_{i=1}^m[\rho_i,\delta_i];\g\right)\\
&=\prod_{i=1}^m[W(J_i,\rho_i;g),W(I_i,\delta_i;\g)]\in G^{(1)}.
\end{align*}
\end{proof}
Configurations show that a group is not soluble with derived length $k$, for a positive integer $k$:
\begin{proposition}
Let $G$ be a group such that $G^{(k)}\neq\{e_G\}$, for a positive integer $k$.  Then, for each generating set $\g$ of $G$, there is a partition $\E$ of $G$, such that the configuration set $\Con(\g,\E)$ cannot arise from a soluble group of derived length $k$.
\end{proposition}
\begin{proof}
Since $G^{(k)}\neq\{e_G\}$, there exists a representative pair, $(J_0,\rho_0)$, in $k$th derivation form such that $W(J_0,\rho_0;\g)\neq e_G$. Set, as in the proof of Lemma \ref{l21}, 
$$S_0:=\{(J,\rho):\, \n J\leq 3\n{J_0},\rho\,\text{is arbitrary}\}.$$
Let $\E$ be any partition which contains $\{e_G\}$ and singletons $\{W(J,\rho;\g)\}$, for $(J,\rho)\in S_0$. Then, by Remark \ref{local homomorphism and refinements}, if $\con gE=\con hF$ for a configuration pair $(\h,\F)$ of a group $H$, then $W(J_0,\rho_0;\h)\neq e_H$. But, $W(J_0,\rho_0;\h)\in H^{(k)}$, for $(J_0,\rho_0)$ is in $k$th derivation form, whence $H^{(k)}\neq\{e_H\}$. 
\end{proof}
We also answer Question 1 in \cite{try} affirmatively:
\begin{theorem}
Let $G$ and $H$ be two groups such that $G\approx H$. Then $G^{(k)}$ and $H^{(k)}$ have same cardinalities, for each positive integer $k$. Furthermore, if $G^{(k)}$ is finite for some positive integer $k$, then we will have $G^{(k)}\cong H^{(k)}$. 
\end{theorem}
\begin{proof}
Let $\g$ be a generating set of $G$. Suppose that $|G^{(k)}|\geq N$ for a positive integer $N$. Then there are representative pairs $(J_i,\rho_i)$, $i=1,\dotsc,N$, in $k$th derivation form, such that $W(J_i,\rho_i;\g)$'s are pairwise distinct. By Lemma \ref{l21}, we can find a generating set $\h$ of $H$ such that $W(J_i,\rho_i;\h)$'s are pairwise distinct, but by previous lemma, $W(J_i,\rho_i;\h)\in H^{(k)}$, so $|H^{(k)}|\geq N$. Therefore, $G^{(k)}$ and $H^{(k)}$ have same cardinalities.\par
Now, suppose that $G^{(k)}$ is finite; consider representative pairs $(J_i,\rho_i)$, $i=1,\dotsc,N$, in $k$th derivation form, such that elements $W(J_i,\rho_i;\g)$'s are non-identity and pairwise distinct in $G^{(k)}$. By Lemma \ref{l21}, we can choose a generating set $\h$ of $H$ such that $W(J_i,\rho_i;\h)$'s are non-identity and pairwise distinct and $\T gh$ is a local homomorphism on 
$$\{e_H\}\cup\{W(J_i,\rho_i;\g):i=1,\dotsc,N\}.$$ 
But, by the first part of the statement, we should have 
$$H^{(k)}=\{e_H\}\cup\{W(J_i,\rho_i;\h):\,i=1,\dotsc,N\}.$$
Therefore, $\T gh|_{H^{(k)}}$ is indeed an isomorphism, and this completes the proof.
\end{proof}
As a consequence of this theorem we have:
\begin{corollary}
Let $G$ and $H$ be two groups such that $G\approx H$. Then $G$ is soluble if and only if  $H$ is soluble. Furthermore, their derived lengths are the same.
\end{corollary}
Now, we will show that being FC can be recovered by configuration sets. The following remark will play a crucial role:
\begin{remark}
\label{remark}
Let $G$ be a group with a generating set $\g$. For $g\in G$, put 
$$\Phi_g:G\rightarrow G,\quad x\mapsto gxg^{-1}$$
and $\inn G:=\{\Phi_g:\,g\in G\}$. It is well-known that $G/Z(G)\cong\inn G$, where $Z(G)$ stands for the center of $G$. For representative pairs $(J_i,\rho_i)$, $i=1,2$, $\Phi_{W(J_1,\rho_1,\g)}\neq \Phi_{W(J_2,\rho_2;\g)}$ if and only if there is a representative pair $(I,\delta)$ such that 
$$\Phi_{W(J_1,\rho_1,\g)}(W(I,\delta;\g))\neq \Phi_{W(J_2,\rho_2;\g)}(W(I,\delta;\g))$$
and one can easily check that the last inequality is equivalent to the following one
\footnotesize
$$W(J_1\oplus I\oplus J_1^{-1},\rho_1\oplus \delta\oplus\sigma_1^{-1};\g)\neq 
W(J_2\oplus I\oplus J_2^{-1},\rho_2\oplus\delta\oplus\sigma_2^{-1};\g).$$
\normalsize
\end{remark}
Now, we assert the main result of the section:
\begin{theorem}
\label{l26}
Let $G$ and $H$ be two groups such that $G\approx H$. Then $\inn G$ and $\inn H$ have same cardinalities. Moreover, if $\inn G$ is finite, then  $\inn G\cong\inn H$.
\end{theorem}
\begin{proof}
Suppose that $|\inn G|\geq N$, for a positive integer $N$. So, there are representative pairs $(J_k,\rho_k)$, $k=1,\dotsc,N$, such that $\Phi_{W(J_k,\rho_k;\g)}$'s are pairwise distinct. By the above remark, for each $k=2,\dotsc,N$, there exist representative pairs, $(I_{k,l},\delta_{k,l})$, $l=1,\dotsc,k-1$, such that 
\footnotesize
\begin{equation}
\label{2}
W(J_k\oplus I_{k,l}\oplus J_k^{-1},\rho_k\oplus \delta_{k,l}\oplus\rho_k^{-1};\g)\neq 
W(J_l\oplus I_{k,l}\oplus J_l^{-1},\rho_l\oplus\delta_{k,l}\oplus\rho_l^{-1};\g)
\end{equation}
\normalsize
Let $\mathfrak{F}$  be a set of below representative pairs, 
$$(J_k,\rho_k)\quad k=1,\dotsc,N$$
along with
\begin{align*}
\begin{cases}
(J_k\oplus I_{k,l}\oplus J_k^{-1},\rho_k\oplus \delta_{k,l}\oplus\rho_k^{-1})\\
(J_l\oplus I_{k,l}\oplus J_l^{-1},\rho_l\oplus\delta_{k,l}\oplus\rho_l^{-1})
\end{cases}\quad k=2,\dotsc,N,\,l=1,\dotsc,k-1.
\end{align*}
Applying Lemma \ref{l21} to $\mathfrak F$, we gain a generating set $\h$ of $H$ such that (\ref{2}) is satisfied for $\h$ instead of $\g$. But, again, Remark \ref{remark} gives that $\Phi_{W(J_k,\rho_k;\h)}$'s are pairwise distinct, so we have $|\inn H|\geq N$, this prove the first part of the Lemma. \par
Now, suppose that $\inn G$ is finite, say 
$$\inn G=\{\Phi_{e_G}\}\cup\{\Phi_{W(J_k,\rho_k;\g)}:\,k=1,\dotsc,N\}.$$
As done earlier, for each $k=1,\dotsc,N$, choose $(I_{k,l},\delta_{k,l})$, $l=1,\dotsc,k-1$, such that 
\footnotesize
\begin{equation*}
W(J_k\oplus I_{k,l}\oplus J_k^{-1},\rho_k\oplus \delta_{k,l}\oplus\rho_k^{-1};\g)\neq 
W(J_l\oplus I_{k,l}\oplus J_l^{-1},\rho_l\oplus\delta_{k,l}\oplus\rho_l^{-1};\g)
\end{equation*}
\normalsize
Construct $\mathfrak{F}$ as above and apply Lemma \ref{l21} to $\mathfrak{F}$ to obtain a generating set $\h$ of $H$, such that (\ref{2}) is satisfied for $\h$ instead of $\g$ and $\T gh$ is a local homomorphism on $$\{e_G\}\cup\{W(J,\rho;\h): (J,\rho) \in\mathfrak{F}\}.$$
Therefore, 
$$\inn H=\{\Phi_{e_G}\}\cup\{\Phi_{W(J_k,\rho_k;\h)}:\,k=1,\dotsc,N\}$$
and $$\Theta:\inn H\rightarrow\inn G,\quad \Phi_{W(J_k,\rho_k;\h)}\mapsto \Phi_{W(J_k,\rho_k;\g)}$$ introduces a desired isomorphism.
\end{proof}

\begin{corollary}
Assume that $G$ and $H$ are two finitely generated groups such that $G$ is an FC-group and $G\approx H$. Then $H$ is an FC-group and the following hold:
\begin{enumerate}
\item $G\times \mathbb{Z}\cong H\times\mathbb{Z}$,
\item $\frac{G}{Z(G)}\cong\frac{H}{Z(H)}$ and $Z(G)\cong Z(H)$,
\item $\frac{G}{G'}\cong\frac{H}{H'}$ and $G'\cong H'$.
\end{enumerate}
\end{corollary}
\begin{proof}
It is proved in \cite{fc} that a finitely generated group $G$ is an FC-group if and only if
$\frac{G}{Z(G)}$ is finite. So, by Theorem \ref{l26} and Remark \ref{remark}, $\frac{G}{Z(G)}\cong\frac{H}{Z(H)}$ and, therefore, $H$ is an FC-group, too.
 \par
If $G$ is a finitely generated FC-group, then $G$ is isomorphic with a subgroup of $\mathbb Z^n\times F$, for some finite group $F$ (see \cite{fc}). Therefore, by \cite[Lemma 1]{try},  $G\times \mathbb{Z}\cong H\times\mathbb{Z}$, $Z(G)\cong Z(H)$ and $G'\cong H'$. Also, \cite[Theorem 2]{try}, gives $\frac{G}{G'}\cong\frac{H}{H'}$.
\end{proof}
The following question is natural:\par
\begin{question}
What we can say about central series of two configuration equivalent groups? Are they equivalent? 
\end{question}
There are non-isomorphic groups $G$ and $H$ such that $G\times \mathbb{Z}\cong H\times\mathbb{Z}$. See the following groups, for instance: 
\begin{align*}
G&:=\langle x,y|\, x^{11}=e_G, y^{-1}xy=x^2\rangle\\
H&:=\langle x,z|\, x^{11}=e_H, z^{-1}xz=x^8\rangle
\end{align*}
In addition, suppose that $zy=yz$, and let $C:=\langle y^7z\rangle$ and $D:=\langle yz^3\rangle$. Then $G\times C\cong H\times D$ (see \cite[Theorem 13]{walker}). Are these two groups configuration equivalent?
\begin{question}
Can the result of Theorem \ref{l26} be stated more sharply by giving a single configuration or set of configurations which shows that the group is not $FC$?
\end{question}

\section{Strong configuration equivalence and Isomorphism}
 In this section we will introduce the notion of strong configuration equivalence and will prove that this type of configuration equivalence leads to isomorphism. First, consider the definition:
 \begin{definition}
 \label{sce}
 We say that two groups $G$ and $H$ are \textit{strong configuration equivalent}, if there exist ordered generating sets $\mathfrak{g}$ of $G$ and $\mathfrak{h}$ of $H$, such that
 \begin{enumerate}
 \item  For each partition $\mathcal{E}$ of $G$ there exists  a partition $\mathcal{F}$ of $H$ such that $\Con(\mathfrak{g},\mathcal{E})=\Con(\mathfrak{h},\mathcal{F})$,
 \item For each partition  $\mathcal{F}$ of $H$ there is a partition $\mathcal{E}$ of $G$ such that
 $\Con(\mathfrak{h},\mathcal{F})=\Con(\mathfrak{g},\mathcal{E})$.
 \end{enumerate}
 In this case, we will write $(G;\mathfrak{g})\approx_s(H;\mathfrak{h})$.\\
 If one of these two conditions is satisfied, say (i), we will say $G$ is \textit{strongly configuration contained} in $H$ and will denote it by 
 $(G;\mathfrak{g})\precsim_s(H;\mathfrak{h})$.
 \end{definition}
 One can easily show, as done in the proof of Lemma \ref{l21}, that 
 \begin{lemma}\label{ll21}
  Let $G$ and $H$ be two groups such that $(G;\mathfrak{g})\precsim_s(H;\mathfrak{h})$. Let $\mathfrak{F}$ be a finite set of representative pairs on $\g$. Then $\T gh$ is a local homomorphism on $\{e_H\}\cup\{W(J,\rho;\h):(J,\rho)\in\mathfrak{F}\}$.
  \end{lemma} 
The following lemma will show that this type of configuration equivalence has the ability to recognize a generating set of a group.

\begin{lemma} 
\label{l22}
If $(G;\g)\precsim_s (H;\h)$, then $\T gh$ is an epimorphism from $H$ onto $G$.
\end{lemma}
\begin{proof}
Suppose that $W(J_0,\rho_0;\g)\neq e_G$. Applying Lemma \ref{ll21} to $\mathfrak F:=\{(J_0,\rho_0)\}$, we conclude that $\T gh$ is a local homomorphism on {\small $\{e_H\}\cup\{W(J_0,\rho_0;\h)\}$}, so, consequently, {\small $W(J_0,\rho_0;\h)\neq e_H$}. This completes the proof.
\end{proof}

 Now, we state the main theorem of this section.
 \begin{theorem}
 Two groups are strongly configuration equivalent if and only if they are isomorphic.
 \end{theorem}
 \begin{proof}
 First suppose that $(G;\mathfrak{g})\approx_s(H;\mathfrak{h})$. By the above lemma, $\T hg$ and $\T gh$ are epimorphism. So, $\T hg:G\rightarrow H$ is an isomorphism.\par 
 Conversely, suppose that $G\overset{\phi}{\cong}H$. Let $\g=(g_1,\dotsc,g_n)$ be a generating set of $G$, and set $\h:=\phi(\g)=(\phi(g_1),\dotsc,\phi(g_n))$. Then $\h$ is a generating set of $H$. If $\E$ is a partition of $G$. Then $\F:=\phi(\E)=\{\phi(E):\, E\in\E\}$ will be a partition of $H$ which satisfies (i) in Definition \ref{sce}. Also, for a partition $\F$ of $H$, $\E:=\phi^{-1}(\F)$ establishes (ii) in the above-mentioned definition.
 \end{proof}

\section{Configuration and isomorphism}
What really makes it difficult to work with configuration equivalence is that it seems that this type of equivalence can not recognize the identity element of a group. In the previous section, this problem was completely resolved by introducing a new type of configuration equivalence.  We now intend to fix this problem partially by defining a special type of configuration pair which is playing an important role in isomorphisms.\par 
Let $G$ be a group and $\g$ be a generating set of $G$. A representative pair $(J,\rho)$ on $\g$ is called \textit{reduced}, if $\rho(k)=\rho(k+1)$, whenever $J(k)=J(k+1)$, for $k<\n J$. It is evident that if $(J,\rho)=(I_1\oplus I_2,\delta_1\oplus\delta_2)$ is reduced, then both representative pairs $(I_k,\delta_k)$'s are reduced, too. 
\begin{definition}
\label{golden}
Let $G$ be a group and $(\g,\E)$ be a configuration pair of $G$ such that $\{e_G\}\in\E$. We call $(\g,\E)$ \textit{golden}, if it can be concluded from the equation $\con gE=\con {g'}{E'}$, for a configuration pair $(\g',\E')$ of $G$, that 
\begin{align}\label{implication} W(J,\rho;\g)\neq e_G\quad \Rightarrow \quad W(J,\rho;\g')E'\cap E'=\emptyset
\end{align}
where $(J,\rho)$ is a reduced representative pair and $E'$ denotes the element of $\E'$ corresponding to $\{e_G\}$.
\end{definition}
The following lemma is exactly what we expect from golden configuration pairs:
\begin{lemma}\label{hofian} Let $(\g,\E)$ be a golden configuration pair of a Hopfian group $G$. Then for each configuration pair $(\g',\E')$ which satisfies $\con gE=\con {g'}{E'}$, $\T g{g'}$ is an automorphism of $G$. Also, if $\{e_G\}\leftrightsquigarrow E'\in\E'$ and $e_G\in E'$, then $(\g',\E')$ is golden too.
\end{lemma}
\begin{proof}
By the implication (\ref{implication}) for a reduced representative pair {\small $(J,\rho)$}, {\small $W(J,\rho;\g')\neq e_G$}, whenever {\small $W(J,\rho;\g)\neq e_G$}, therefore, $\T g{g'}$ is an epimorphism from $G$ onto $G$. But $G$ is Hopfian, hence $\phi:=\T g{g'}$ is indeed an automorphism. 
\par Now, assume that $\{e_G\}\leftrightsquigarrow E'\in\E'$ and $E'$ contains $e_G$. If $W(J,\rho;\g')\in E'$ for a reduced representative pair $(J,\rho)$, then $\phi(W(J,\rho;\g'))=W(J,\rho;\g)=e_G$, whence $W(J,\rho;\g')=e_G$, so $E'=\{e_G\}$. If $\con {g'}{E'}=\con {g''}{E''}$, for a configuration pair $(\g'',\E'')$ of $G$, and $\{e_G\}\leftrightsquigarrow E''$, then for each reduced pair $(J,\rho)$, 
\begin{align*}
W(J,\rho;\g')\neq e_G&\quad \Rightarrow \quad W(J,\rho;\g)=\phi(W(J,\rho;\g'))\neq e_G\\
&\quad \Rightarrow \quad W(J,\rho;\g'')E''\cap E''=\emptyset
\end{align*}
\end{proof}
\begin{example}\label{groups with golden configuration pairs}
Below, we've listed some groups which have a golden configuration pair:
  \begin{enumerate}
    \item All non-Abelian free groups have a golden configuration pair. Consider a generating set $\f=(f_1,\dotsc,f_n)$ of $\textbf{F}_n$. Set 
$$\E=\{E_0, E_k, E_{-k};\,k=1,\dotsc,n\}$$
where $E_0=\{e_{\textbf{F}_n}\}$, and 
\begin{align*}
E_k&=\{\text{reduced words starting with $f_k$}\}\\
E_{-k}&=\{\text{reduced words starting with $f_k^{-1}$}\}
\end{align*}
for $k=1,\dotsc,n$. One can easily verify that for $k\in\{1,\dotsc,n\}$, 
\begin{align}\label{0}
f_k(\textbf{F}_n\setminus E_{-k})=E_k\quad \text{and}\quad
f_kE_{-k}=\textbf{F}_n\setminus E_{k}
\end{align}
 If $H$ is a group with a configuration pair $(\h,\F)$ such that $\con fE=\con hF$. Then $\h$ and $\F$ can be displayed as 
\begin{align*}
\h&=(h_1,\dotsc,h_n)\\
\F&=\{F_0, F_k,F_{-k};\,k=1,\dotsc,n\},\quad e_H\in F_0
\end{align*}
where 
$$F_0\leftrightsquigarrow E_0, F_k\leftrightsquigarrow E_k, F_{-k}\leftrightsquigarrow E_{-k}, k=1,\dotsc,n.$$
By Lemma \ref{important} and relations (\ref{0}), for $k\in\{1,\dotsc,n\}$, the following relations will hold:
\begin{align*}
h_k(H\setminus F_{-k})=F_k\quad \text{and}\quad
h_kF_{-k}=H\setminus F_{k}
\end{align*}
Considering these relations, it may be concluded that for each reduced representative pair $(J,\rho)$ on $\h$, 
$$W(J,\rho;\h)F_0\subseteq F_{\rho(1)J(1)}.$$
So, $(\mathfrak f,\E)$ is a golden configuration pair of $\textbf F_n$ (see \cite[Proposition 6.1]{arw} for details). 

\item Let $\mathbb Z$ be the group of integers, $n$ be a positive integer and $F$ be a finite group. Then all groups on the form $\mathbb Z^n\times F$ have a golden configuration pair. Indeed, suppose that $F=\{x_0=e_F,x_1,\dotsc,x_m\}$ is an arbitrary finite group and $n\in\mathbb N$. Let $\g=(g_1,\dotsc,g_{n+m})$, where 
\begin{align*}
g_i&=(\e_i,e_F),\quad i=1,\dotsc,n\\
g_{n+j}&=(\mathbf o,x_j),\quad j=1,\dotsc,m.
\end{align*}
where $\mathbf o$ is the neutral element of $\mathbb Z^n$, and $\e_i$ is the element of $\mathbb Z^n$, whose only nonzero component, $i$th one, is 1. \par
Let $\Sigma$ be the set of all functions from $\{1,\dotsc,n\}$ into $\{-1,0,1\}$. Set 
$$E(\tau,j)=\tau(1)\mathbb N\times\dotsb\times\tau(n)\mathbb N\times\{x_j\}$$
 for $\tau\in\Sigma$ and $j=0,1,\dotsc,m$. Consider the $\sigma$-algebra, $\A$, generated by sets $\{g_i\}$, $\{g_ig_j\}$, $1\leq i,j\leq n$ and $E(\tau,j)$, $\tau\in\Sigma$ and $j=0,1,\dotsc,m$. Then $(\g,\atom A)$ is a golden configuration pair. By the proof of \cite[Theorem 3.5]{ary}, the reader can certify the correctness of this claim. In particular, all finite and all Abelian groups have a golden configuration pair.

\item The infinite dihedral group, $D_\infty=\langle x,y: x^2=y^2=1\rangle$, has a golden configuration pair. Let $\g=(x,y)$, and $\E=\{E_k:k=1,\dotsc,5\}$, where 
$$E_1=\{e_{D_\infty}\},\,E_2=\{x\},\,E_3=\{y\}$$ 
and 
\begin{align*}
E_4&=\{g_1g_2\dots g_n: n\in\mathbb N, n>1, g_1=x, g_i\in\{x,y\}, g_i\neq g_{i-1}, i=2,\dotsc,n\}\\
E_5&=\{g_1g_2\dots g_n: n\in\mathbb N, n>1, g_1=y, g_i\in\{x,y\}, g_i\neq g_{i-1}, i=2,\dotsc,n\}
\end{align*}
By \cite[Example 3.7]{ary}, it can be seen that $(\g,\E)$ is a golden configuration pair.
  \end{enumerate}
\end{example}

Let $G$ and $H$ be two groups. Consider partitions $\E=\{E_1,\dotsc,E_r\}$ of $G$, $\F=\{F_1,\dotsc,F_r\}$ of $H$, and their refinements
$$
\E'=\{E'_1,\dotsc,E'_s\},\text{and}\,\F'=\{F'_1,\dotsc,F'_s\}.
$$
We say that these two refinements $\E'$ and $\F'$ are \textbf{similar} and write $(\E',\E)\sim(\F',\F)$, if 
$$ \{l : E_k\cap E'_l\neq\emptyset\}=\{l : F_k\cap F'_l\neq\emptyset\}\quad(k=1,\dotsc,r).$$
In other words, if $E_k=\bigcup_{j=1}^tE'_{i_j}$, then we have $F_k=\bigcup_{j=1}^tF'_{i_j}$. \par
Note that it is implicit in the definition of similarity that similar
partitions have equal numbers of sets.
\begin{lemma}\label{3th to 4th part}
Let $\E'=\{E'_1,\dotsc,E'_s\}$ be a refinement of a partition $\E=\{E_1,\dotsc,E_r\}$ of $G$. For a partition $\F'=\{F'_1,\dotsc,F'_s\}$ of $H$, there exists a partition $\F=\{F_1,\dotsc,F_r\}$ of $H$ such that $(\E',\E)\sim(\F',\F)$.
\end{lemma}
\begin{proof}
It is enough to set 
$$F_k=\bigcup\{F'_l:\, E_k\cap E_l'\neq\emptyset\}$$
for $k=1,\dotsc,r$.
\end{proof}

An important feature of similar refinements is presented below:
\begin{lemma}
\label{similarity}
let $G$ and $H$ be two groups. Assume that $(\g,\E)$ and $(\h,\F)$ are two configuration pairs for $G$ and $H$, respectively, and let $\E'$ and $\F'$ be their similar refinements such that $\con g{E'}=\con h{F'}$. Then $\con gE=\con hF$.
\end{lemma}
\begin{proof}
Without loss of generality, let 
\begin{gather*}
\E=\{E_1,\dotsc,E_m\}\quad\text{and}\quad \E'=\{K_1,\dotsc,K_m,K_{m+1}\}
\end{gather*}
where $K_i=E_i$, $i=1,\dotsc,m-1$, and $K_m\cup K_{m+1}=E_m$.\par
Now only note that if $C=(c_0,c_1,\dotsc,c_n)$ belongs to $\con g{E'}$; by changing components which are $m$ or $m+1$ into $m$, we will obtain a configuration $\widehat{C}$ in $\con gE$.\par
Conversely, assume that $\con gE$ contains a configuration $C=(c_0,c_1,\dotsc,c_n)$.  Let $(x_0,x_1,\dotsc,x_n)$ have the configuration $C$. Now, replace components $c_i=m$ with $m$ or $m+1$ depending on $x_i\in K_m$ or $x_i\in K_{m+1}$, respectively, to obtain a configuration $\tilde C$ in $\con g{E'}$. 
\par The proof will be complete by noting that, for $C\in\con gE$, $\widehat{\tilde C}=C$.
\end{proof}
 
The following lemma is of particular importance:
\begin{lemma}
\label{refinement}
Let $G$ be a group with a golden configuration pair $(\g,\E)$. Then for each refinement $\E'$ of $\E$, the configuration pair $(\g,\E')$ is golden too. 
\end{lemma}
\begin{proof}
Assume that $(\mathfrak{x},\mathcal K')$ is a configuration pair of $G$ such that 
$\con \g{\E'}=\Con(\mathfrak{x},\mathcal K')$. Assume that $\{e_G\}\leftrightsquigarrow K'\in \mathcal K'$. Make $\mathcal K'$ coarser to gain a partition $\mathcal K$ such that $(\E',\E)\sim(\mathcal K',\mathcal K)$ (see lemma \ref{3th to 4th part}). So, If $\{e_G\}\leftrightsquigarrow K\in \mathcal K$, then $K=K'$. By Lemma \ref{similarity}, we have $\con \g\E=\Con(\mathfrak{x},\mathcal K)$, and this completes the proof.
\end{proof}
\begin{lemma}
\label{recognize}
Let $G$ and $H$ be two groups with $G\approx H$. Assume that $G$ is Hopfian with a golden configuration pair $(\g,\E)$ and that $(\h,\F)$ is a configuration pair for $H$ such that $\con gE=\con hF$. Then 
\begin{itemize}
\item[(a)] If $\{e_G\}\leftrightsquigarrow F\in\F$, then $F$ is a singleton set. 
\item[(b)] Let $\F'$ be a refinement of $\F$. Then there exists a partition $\E'$ of $G$ such that $(\g,\E')$ is a golden configuration pair and $\con h{F'}=\con g{E'}$. 
\end{itemize}
 \end{lemma}
 \begin{proof}
(a)Assume to contrary that, $F$ is not a singleton set, so we can write $F=F_1\cup F_2$, for nonempty sets $F_1$ and $F_2$. Consider the following refinement of $\F$,
$$\mathcal K:=\{F_1,F_2\}\cup(\F\setminus\{F\}).$$
There is a configuration pair $(\g',\mathcal L')$ such that $\con {g'}{L'}=\con hK$. Suppose that $F_i\leftrightsquigarrow L_i\in\mathcal L'$, $i=1,2$. Let partition $\mathcal L$ be such that $(\mathcal L',\mathcal{L})\sim (\mathcal{K}',\mathcal{K})$. Lemma \ref{similarity} implies that 
$$\con {g'}L=\con hF=\con gE.$$
But $(\g,\E)$ is golden, hence by Lemma \ref{hofian} and Remark \ref{remark 1}, we can assume that $(\g',\mathcal L)$ is golden, so we should have $\{e_G\}=L_1\cup L_2$ and this is impossible.\par 
(b) Now, let $\F'$ be a refinement of $\F$. By $G\approx H$, there exists a configuration pair 
$(\mathfrak x,\mathcal P')$ of $G$ such that $\con h{F'}=\Con(\mathfrak x,\mathcal P')$. Assume $\mathcal P$ is finer than $\mathcal P'$ with $(\mathcal P',\mathcal P)\sim(\F',\F)$. Hence, by lemma \ref{similarity} 
\begin{gather} 
\label{rrrrrr}
\con x{P}=\con hF=\con gE
\end{gather}
so, $\psi:=\T gx$ is an automorphism of $G$. Now, put 
$$\E':=\psi^{-1}(\mathcal P')=\{\psi^{-1}(P'):\,P'\in\mathcal P'\}.$$
We have 
$$\con g{E'}=\Con(\psi(\g),\psi(\E'))=\con x{P'}=\con h{F'}$$
and by Lemma \ref{hofian}, we can assume that $(\g,\E')$ is golden.
 \end{proof}
Now, we will state and prove the main theorem of this section.
 \begin{theorem}
 \label{last}
  Let $G$ be a Hopfian group with a golden configuration pair and $H$ be a group such that $G\approx H$. Then $G$ is finitely presented if and only if $H$ is finitely presented, and in the case that $G$ or $H$ is a finitely presented group, we have $G\cong H$.
  \end{theorem}
  \begin{proof}
  Let $(\g,\E)$ be a golden configuration pair of the Hopfian group $G$. 
   \par Suppose that $G$ is finitely presented and put $\{W(J_i,\rho_i;\g), i=1,\dotsc,m\}$ for its set of defining relators. Set $\mathfrak F=\{(J_i,\rho_i), i=1,\dotsc,m\}$. By  Lemma \ref{refinement}, we can assume that $\E$ contains $\{e_G\}$ and singletons $\{W(J,\rho;\g)\}$, $(J,\rho)\in S_0$, where $S_0$ is defined as in the proof of Lemma \ref{l21}. Now, consider $\con gE=\con hF$, for a configuration pair $(\h,\F)$ of $H$. Hence, according to Lemma \ref{recognize}, part (a), we have $\{e_H\}\in\F$. Also, by Remark \ref{local homomorphism and refinements}, we have $W(J_i,\rho_i;\h)=e_H$, $i=1,\dotsc,m$.\\
    We claim that $\{W(J_i,\rho_i;\h), i=1,\dotsc,m\}$ is a set of defining relators, because if it is not, then we can find a relator in $H$, say $W(I,\delta;h)$, which can not be obtained from $\{W(J_i,\rho_i;\h), i=1,\dotsc,m\}$. But, by Lemma \ref{recognize}, (b), and using Remark \ref{local homomorphism and refinements} again, we have $W(I,\delta;\g)=e_G$ and this contradicts the fact that $\{W(J_i,\rho_i;\g), i=1,\dotsc,m\}$ is a set of defining relators of $G$. \\
  So, $G$ and $H$ are two groups with the same sets of defining relators, and therefore $G\cong H$, by \cite[Theorem 1.1]{combin}. 
  \par Now, let $(\h,\F)$ be a configuration pair such that $\con gE=\con hF$. Put $\{W(J_i,\rho_i;\h), i=1,\dotsc,m\}$ for a set of defining relators of $H$, and consider a representative pair $(J_0,\rho_0)$ such that $W(J_0,\rho_0;\h)\neq e_H$. Appealing once more to Remark \ref{local homomorphism and refinements}, and using part (b) of Lemma \ref{recognize} again, we can assume, without loss of generality that $W(J_i,\rho_i;\g)=e_G$, $i=1,\dotsc,m$ and $W(J_0,\rho_0;\g)\neq e_G$. Hence, \cite[Theorem 1.1]{combin} implies $G\cong H$.
  \end{proof}
  By previous theorem and Example \ref{groups with golden configuration pairs}, we have:
  \begin{corollary}
  The following hold:
  \begin{itemize}
  \item[(a)] If $G\approx\mathbb{F}_n$, then $G\cong\mathbb{F}_n$.
  \item[(b)] If $G\approx\mathbb Z^n\times F$, where $F$ is an arbitrary finite group, then $G\cong \mathbb Z^n\times F$.
  \item[(c)] If $G\approx D_\infty$, then $G\cong D_\infty$.
  \end{itemize}
  \end{corollary}
 Theorem \ref{last} leads us to the following questions:\\ 
\begin{question} (1) Is there a finitely generated group without a golden configuration pair?\par
(2) Does each Hopfian group have a golden configuration pair? \par
(3) How about finitely presented Hopfian groups? Do they have a golden configuration pair?
\end{question}

\end{document}